\documentclass{amsart}

\usepackage{graphicx}

\newtheorem*{theorem*}{Theorem}
\newtheorem{theorem}{Theorem}[section]
\newtheorem{lemma}[theorem]{Lemma}
\newtheorem{proposition}[theorem]{Proposition}

\begin{document}

\title{Non left-orderable surgeries on negatively twisted torus knots}

\author{Kazuhiro Ichihara}
\address{Department of Mathematics, 
College of Humanities and Sciences, Nihon University,
3-25-40 Sakurajosui, Setagaya-ku, Tokyo 156-8550, Japan}
\email{ichihara@math.chs.nihon-u.ac.jp}
\thanks{Ichihara was partially supported by JSPS KAKENHI Grant Number 26400100.}

\author{Yuki Temma}
\address{MY.TECH CO., LTD.}

\begin{abstract}
We show that certain negatively twisted torus knots admit Dehn surgeries yielding 3-manifolds with non left-orderable fundamental groups. 
\end{abstract}



\date{\today}

\maketitle

\section{Introduction}

In \cite{BGW}, Boyer, Gordon and Watson raised the following conjecture, which is now called the \textit{L-space conjecture}: 
An irreducible rational homology 3-sphere is an L-space if and only if its fundamental group is not left-orderable. 
Here 
a rational homology 3-sphere is called an \textit{L-space} if rk$\widehat{HF}(M) = |H_1(M, \mathbb{Z})|$ holds, 
and 
a non-trivial group $G$ is called \textit{left-orderable} if there exists a strict total order $<$ on $G$ such that if $g < h$, then $f g < f h$ holds for any $f , g , h \in G$. 
See \cite{BGW} for more details. 

One of the known approaches to the conjecture is by using Dehn surgery, for it gives a simple way to construct many L-spaces at once. 
For example, in \cite{OS}, it is shown that a knot $K$ in the 3-sphere $S^3$ admits a Dehn surgery yielding an L-space, then any Dehn surgery on $K$ along a slope $r$ with $r \ge 2 g(K) -1$ yields an L-space, where $g(K)$ denotes the genus of $K$. 

Thus it is natural to ask if Dehn surgeries on a knot yielding 3-manifolds with non left-orderable fundamental groups give L-spaces. 

There are several known results on this line. 
For example, recall that finite groups cannot be left-orderable, and 
3-manifolds with finite fundamental groups are known to be L-spaces. 
Thus a good target for such study is the class of knots admitting Dehn surgeries creating 3-manifolds with finite fundamental groups. 
A well-known class of the knots with such surgeries is that of twisted torus knots, originally given by Dean \cite{D}. 

Here a \textit{(positively) twisted torus knot} in $S^3$ is defined as the knot obtained from a torus knot by adding (positive) full twists along an adjacent pair of strands. 
See Figure \ref{twisted} for example. 

\begin{figure}[htb]
\centering
\includegraphics[bb = 0 0 415 183, scale=.3]{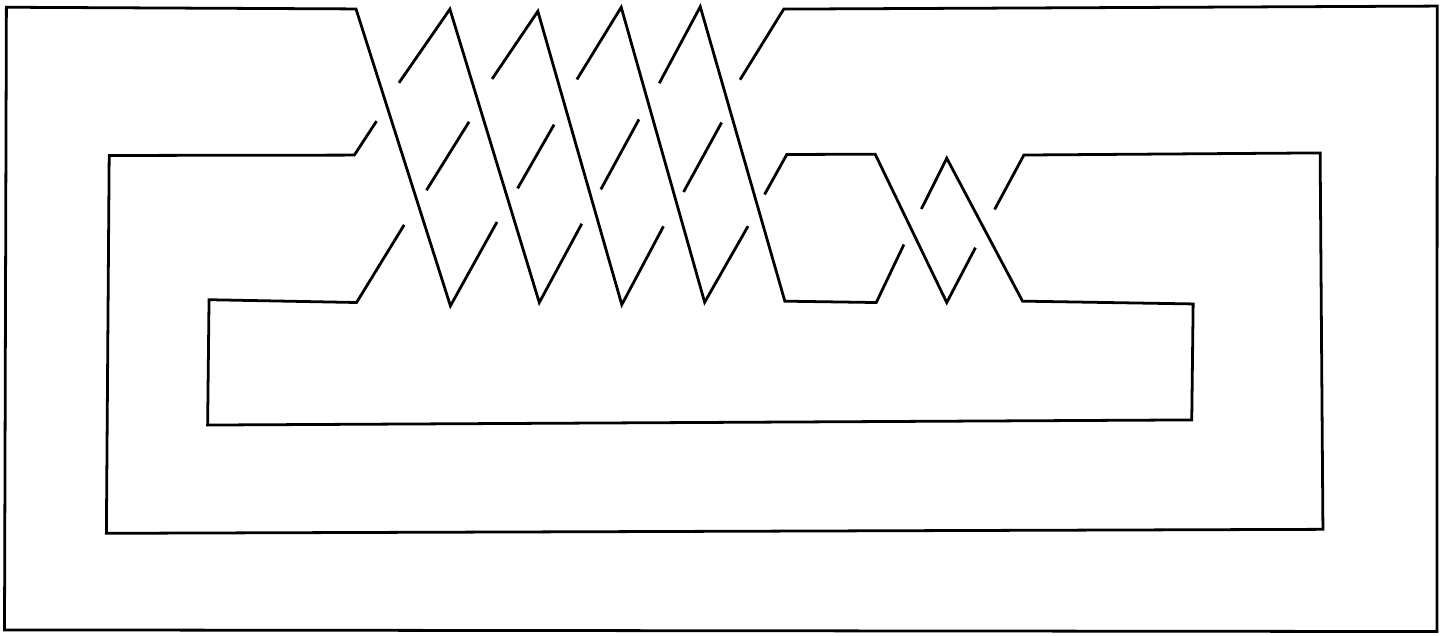}
\caption{$1$-twisted torus knot of type $(3, 5)$}\label{twisted}
\end{figure}

In fact, Clay and Watson considered Dehn surgeries on positively twisted torus knots in \cite{CW}. 
Precisely, they showed that $r$-surgery on a positively $s$-twisted $(3, 3k + 2)$-torus knot gives a 3-manifold with non left-orderable fundamental group if (1) $s \ge 0$ and $k=1$, or (2) $s=1$ and $k \ge 0$. 

This can be comparable to the result by Vafaee \cite{V} showing that such positively twisted torus knots admit L-space surgery. 
Precisely, he showed that $r$-surgery on 
$(p, k p \pm 1)$-torus knots with $s$ full twists on $u$ adjacent strands 
gives an L-space if (1) $u=p-1$, (2) $u=p-2$ and $s=1$, or (3) $u=2$ and $s=1$, for $p \ge 2$, $k \ge 1$, $s>0$, and $0 < u<p$. 
In the case of $p=3$, a positively $2$-twisted $(3, q)$-torus knot ($q>0$, $s \ge 1$) is of genus $q+s-1$, and so, it follows that $r$-surgery on that knot gives a 3-manifold with non left-orderable fundamental group if $r \ge 2q + 2s -3$.

In \cite{IT}, by refining the argument of Nakae \cite{N} for pretzel knots, we have extended the result of Clay and Watson to show that 
$r$-surgery on a positively $s$-twisted torus knot of type $(3, 3v + 2)$ with $s, v \ge 0$ yields a closed 3-manifold with non left-orderable fundamental group if $ r \ge 3 ( 3v+2) + 2s $. 

Soon after, Christianson, Goluboff, Hamann, and Varadaraj gave further extensions in \cite{CGHV}. 
They actually show that $r$-surgery on $(p, k p \pm 1)$-torus knots with $s$ full twists on $u$ adjacent strands gives a closed 3-manifold with non left-orderable fundamental group 
(1) $u=p-1$ and $r \ge p ( pk \pm 1) + (p-1)^2 s$, or (2) $u=p-2$, $s=1$ and $r \ge p ( pk \pm 1) + (p-2)^2$, for $p \ge 2$, $k \ge 1$, $s>0$, and $0 < u<p$. 
As an immediate corollary, it follows that $r$-surgery on a positively $s$-twisted torus knot of type $(3, q)$ with $u, q \ge 0$ yields a closed 3-manifold with non left-orderable fundamental group if $ r \ge 3 q + 4 s $ for $s>1$ and if $ r \ge 3 q + 1$ for $s=1$. 

\medskip

We remark that all the above results concern only positively twisted torus knots. 
Thus it seems natural to ask what happens for \textit{negatively} twisted torus knots. 
In this paper, we show the following. 

\begin{theorem}\label{thm1}
Let $K$ be the $(-1)$-twisted $(3, 3v + 2)$-torus knot with $v \ge 0$. 
Then $r$-surgery on $K$ yields a 3-manifold with non left-orderable fundamental group if $r \ge 3 ( 3v+2) -2$. 
\end{theorem}

Actually, the $(-1)$-twisted $(3, 3v + 2)$-torus knot is equivalent to the $1$-twisted $(3, 3v + 1)$-torus knot with $v \ge 0$. 
Thus, the result in \cite{CGHV} guarantees that $r$-surgery on the knot yields a 3-manifold with non left-orderable fundamental group if $r \ge 3 ( 3v+2) +1$. 
That is, our theorem above can improve their bound slightly. 

\section{Proof}

We start with recalling our technical main result in \cite{IT}. 

\begin{theorem*}[{\cite[Theorem 1.1]{IT}}]\label{thm}
Let $K$ be a knot in a closed, connected 3-manifold $M$. 
Suppose that the knot group $\pi_1(M-K)$ admits the presentation  
\[
\langle a, b \ \vert \ (w_1 a^m w_1^{-1}) b^{-r} (w_2^{-1} a^n w_2) b^{r-k} \rangle
\]
Here $w_1, w_2$ are arbitrary words 
with $m, n\ge 0$, $r\in \mathbb{Z}$, $k\ge 0$. 
Suppose further that $a$ represents a meridian of $K$ and 
$a^{-s} w a^{-t}$ represents a longitude of $K$ with $s, t\in \mathbb{Z}$ and $w$ is a word which excludes $a^{-1}$ and $b^{-1}$.
Then if $q\ne 0$ and $p/q\ge s+t$, then Dehn surgery on $K$ along the slope corresponding to $p/q$ with respect to the meridian-longitude system yields a closed 3-manifold with non left-orderable fundamental group.
\end{theorem*}

Thus, to prove Theorem~\ref{thm1}, we show the following. 

\begin{proposition}\label{prop}
Let $K$ be the $u$-twisted torus knot of type $(3,3v+2)$ in $S^3$ with $u\in \mathbb{Z}, v \ge 0$. 
Then the knot group $\pi_1(S^3-K)$ admits the presentation 
\[
\langle a, b \ \vert \ (w_1 a^m w_1^{-1}) b^{-r} (w_2^{-1} a^n w_2) b^{r-k} \rangle
\]
with $m=n=1$, $r=u+1$, $k=1$, $w_1 = (ba)^{v+1}$, $w_2=(ba)^v$. 
Furthermore the generator $a$ represents a meridian of $K$ and 
the preferred longitude of $K$ is represented as $a^{-s} w a^{-t}$ with 
$s=2u+3(3v+2)+1$, $t=-1$ and $w = ((ba)^v b^{u+1} )^2 (ba)^v b$. 
\end{proposition}

In the following, we will obtain a \textit{Wirtinger presentation} of such a knot group. 
Note that, in \cite{CW, CGHV, IT}, all the presentations for knot groups are obtained by using Heegaard splitting of $S^3$. 

Let $L$ be the three component link, and $D$ the diagram of $L$ illustrated in Figure~\ref{3complink1}. 

\begin{figure}[htb]
\begin{center}
\includegraphics[bb = 0 0 346 404, scale=.3]{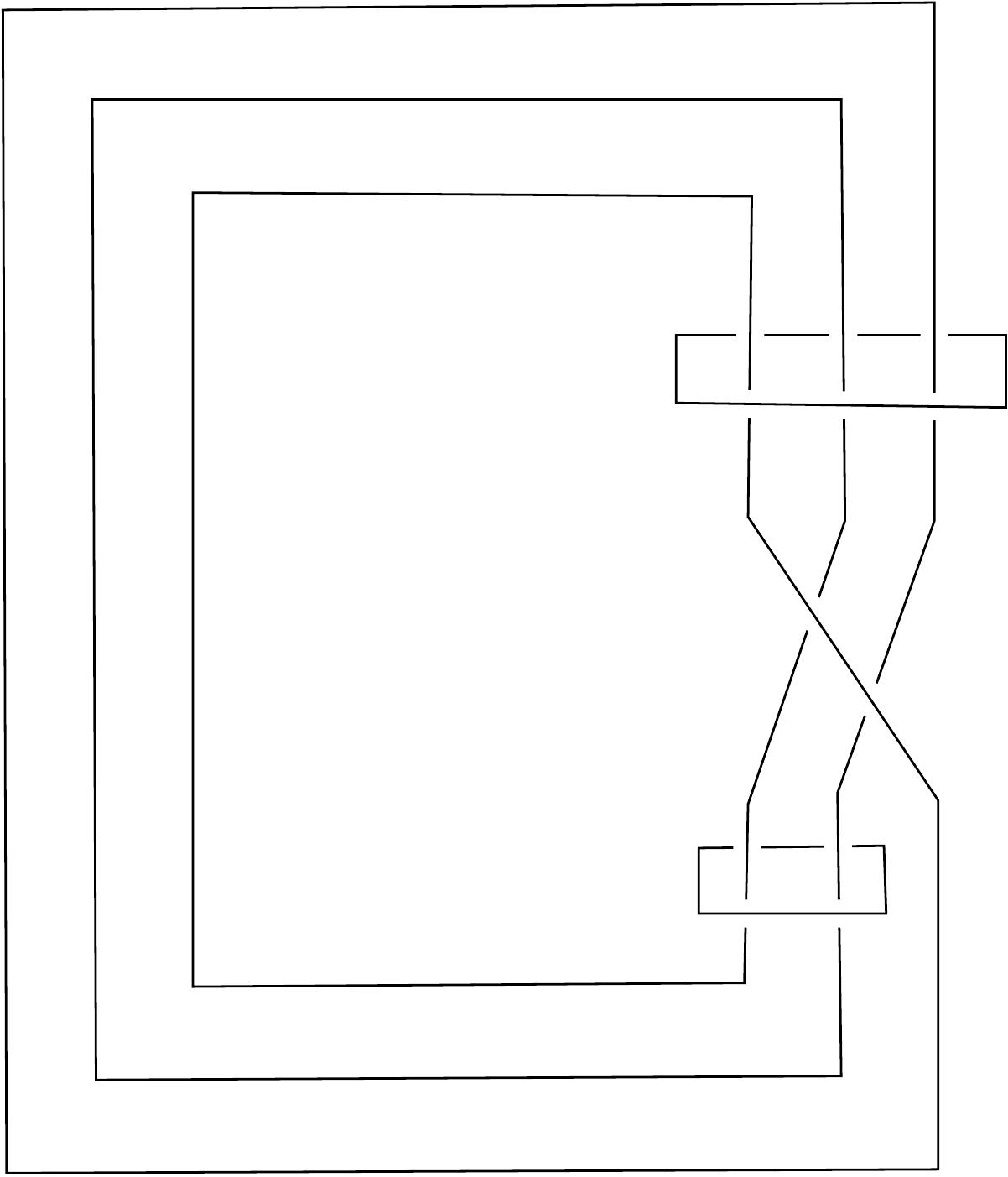}
\caption{three component link 1}\label{3complink1}
\end{center}
\end{figure}

As in Figure~\ref{3complink3}, we label $\alpha, \beta, \gamma, \xi, \psi, 
\delta_1, \delta_2, \delta_3, \delta_4, \delta_5, \delta_6, \delta_7$ to the arcs in the diagram, and $P_1, \cdots, P_{12}$ to the twelve crossings in the diagram,. 

\begin{figure}[htb]
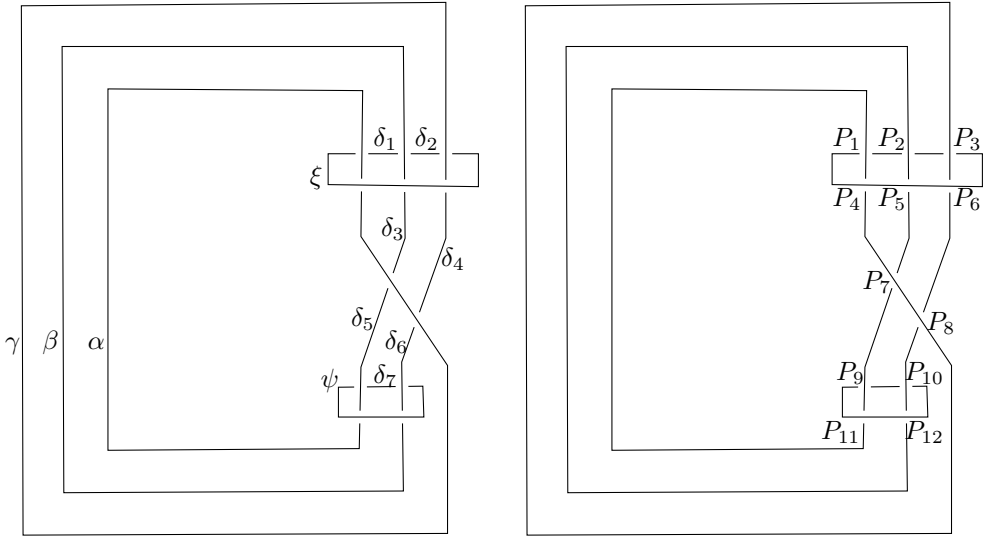
\centering
  {\unitlength=1cm
  \begin{picture}(10,7)
   \put(-1.5,0){\includegraphics[bb = 0 0 346 404, scale=.5]{3complink1.pdf}}
   \put(5.2,0){\includegraphics[bb = 0 0 346 404, scale=.5]{3complink1.pdf}}
  \put(-1.7,2.5){$\gamma$}
  \put(-1.2,2.5){$\beta$}
  \put(-0.6,2.5){$\alpha$}
  \put(2.35,4.7){$\xi$}
  \put(2.5,2){$\psi$}
  \put(3.2,5.2){$\delta_1$}
  \put(3.75,5.2){$\delta_2$}
  \put(3.3,4){$\delta_3$}
  \put(4.1,3.6){$\delta_4$}
  \put(2.9,2.8){$\delta_5$}
  \put(3.35,2.5){$\delta_6$}
  \put(3.2,2.05){$\delta_7$}
  \put(9.3,5.2){$P_1$}  
  \put(9.9,5.2){$P_2$}  
  \put(10.9,5.2){$P_3$}  
  \put(9.3,4.4){$P_4$}  
  \put(9.9,4.4){$P_5$}  
  \put(10.9,4,4){$P_6$}  
  \put(9.7,3.3){$P_7$}
  \put(10.55,2.75){$P_8$}
  \put(9.35,2.08){$P_9$}  
  \put(10.26,2.08){$P_{10}$}  
  \put(9.15,1.3){$P_{11}$}  
  \put(10.26,1.3){$P_{12}$}  
\end{picture}}
\caption{three component link}\label{3complink3}
\end{figure}

Next we construct relators for the crossing of $D$ as follows. 
\begin{align*}
& P_1: \ & \xi\alpha\delta_1^{-1}\alpha^{-1} \\
& P_2: \ & \delta_1\beta\delta_2^{-1}\beta^{-1} \\
& P_3: \ & \delta_2\gamma\xi^{-1}\gamma^{-1} \\
& P_4: \ & \xi^{-1}\alpha\xi\gamma^{-1} \\
& P_5: \ & \xi^{-1}\beta\xi\delta_3^{-1} \\
& P_6: \ &        \xi^{-1}\gamma\xi\delta_4^{-1} \\
& P_7: \ & \delta_5^{-1}\gamma\delta_3\gamma^{-1} \\
& P_8: \ & \delta_6^{-1}\gamma\delta_4\gamma^{-1} \\
& P_9: \ & \psi\delta_5\delta_7^{-1}\delta_5^{-1} \\
& P_{10}: \ & \delta_7\delta_6\psi^{-1}\delta_6^{-1} \\
& P_{11}: \ & \psi^{-1}\delta_5\psi\alpha^{-1} \\
& P_{12}: \ & \psi^{-1}\delta_6\psi\beta^{-1} 
\end{align*}
Moreover, the following are obtained from the above. 
\begin{align*}
\delta_1 &= \alpha^{-1}\xi\alpha  & ( P_1 )\\
\delta_2 &= \gamma\xi\gamma^{-1}  & ( P_3 )\\
\delta_3 &= \xi^{-1}\beta\xi  & ( P_5 ) \\
\delta_4 &= \xi^{-1}\gamma\xi  & ( P_6 )\\
\delta_5 &= \gamma\delta_3\gamma^{-1} = \gamma\xi^{-1}\beta\xi\gamma^{-1}  & ( P_7 ) \\
\delta_6 &= \gamma\delta_4\gamma^{-1} = \gamma\xi^{-1}\gamma\xi\gamma^{-1}  & ( P_8 )
\end{align*}
Moreover, from the relations at $P_{11}$ and $P_{12}$, 
we have $\delta_5 = \psi \alpha \psi^{-1}$ and $\delta_6 = \psi \beta \psi^{-1}$. 
From these, together with the relations at $P_{9}$ and $P_{10}$, we have 
$\psi^{-1}\beta^{-1}\alpha^{-1}\psi\alpha\beta$. 
From these, by erasing $\delta_1, \cdots, \delta_7$, we have the following lemma consequently. 
Here let $l_1$ be an upper trivial component, $l_2$ be a lower trivial component, and $l_0$ be the other component in Figure \ref{3complink1}. 

\begin{lemma}\label{lem}
For the three component link $L$ in Figure \ref{3complink1}, the link group $\pi_1(S^3-L)$ admits the presentation 
\[
\left\langle \alpha, \beta, \gamma, \xi, \psi \ \left\vert \ 
\begin{array}{l}
\xi^{-1}\gamma^{-1}\beta^{-1}\alpha^{-1}\xi\alpha\beta\gamma, 
\xi^{-1}\alpha\xi\gamma^{-1}, 
 \psi^{-1}\gamma\xi^{-1}\beta\xi\gamma^{-1}\psi\alpha^{-1}, \\
\psi^{-1}\gamma\xi^{-1}\gamma\xi\gamma^{-1}\psi\beta^{-1}, 
\psi^{-1}\beta^{-1}\alpha^{-1}\psi\alpha\beta
\end{array} \right. 
\right\rangle
\]
Furthermore, for the component $l_1$ of $L$, a meridian is  represented by $\xi$ and the preferred longitude is represented by $\alpha\beta\gamma$. For the component $l_2$ of $L$, a meridian is represented $\psi$ and the preferred longitude is represented by $\alpha\beta$. 
Also, for the component $l_0$, a meridian is represented by $\alpha$ and a longitude is represented by $\xi\xi\gamma^{-1}\psi\xi\gamma^{-1}\psi$. 
\end{lemma}

In fact, $u$-twisted torus knots of type $(3, 3v+2)$ are obtained from the link $L$ by adding $u$ full twists along the component $l_2$, and $v$ full twists along the component $l_1$.
We note that adding $v$ full twists along $l_1$ is realized by the Dehn surgery on $l_1$ along the slope $-\frac{1}{v+1}$. 
Also note that adding $u$ full twists along $l_2$ is realized by the Dehn surgery on $l_2$ along the slope $-\frac{1}{u}$. 
Hence, to obtain a presentation of the knot group for a $u$-twisted torus knots of type $(3, 3v+2)$, 
it is sufficient to add relations $\xi(\alpha\beta\gamma)^{-v-1}=1$ and $\psi(\alpha\beta)^u=1$ to the presentation given in Lemma \ref{lem}. 

\begin{proof}[Proof of Proposition \ref{prop}]
Let $K$ be a $u$-twisted torus knot of type $(3, 3v+2)$. 
By Lemma \ref{lem} and the argument above, the knot group of $K$ admits the presentation 
\[
\left\langle \alpha, \beta, \gamma, \xi, \psi \ \left\vert \ 
\begin{array}{l}
\xi^{-1}\gamma^{-1}\beta^{-1}\alpha^{-1}\xi\alpha\beta\gamma, 
\xi^{-1}\alpha\xi\gamma^{-1}, 
 \psi^{-1}\gamma\xi^{-1}\beta\xi\gamma^{-1}\psi\alpha^{-1}, \\
\psi^{-1}\gamma\xi^{-1}\gamma\xi\gamma^{-1}\psi\beta^{-1}, 
\psi^{-1}\beta^{-1}\alpha^{-1}\psi\alpha\beta, 
\xi(\alpha\beta\gamma)^{-v-1}, \psi(\alpha\beta)^u
\end{array} \right. 
\right\rangle
\]

We only have to prove that this presentation can be transformed to the presentation in \cite[Theorem 1.1]{IT}. 
Now we set $g:=\xi\gamma^{-1}$, $h:=\alpha\beta\gamma$. 
Then the generators $\alpha, \beta, \gamma, \xi, \psi$ of the presentation above are described by  using $g$ and $h$ as follows.
\begin{align*}
  \xi &= (\alpha\beta\gamma)^{v+1} =h^{v+1} \\
  \gamma &= g^{-1}\xi=g^{-1}h^{v+1} \\
  \psi &= (\alpha\beta)^u=(h^{-v}g)^u \\
  \alpha &= \xi\gamma\xi^{-1}=h^{v+1}g^{-1} \\
  \beta &= gh^{-v-1}h^{-v}g =gh^{-2v-1}g
\end{align*}
Thus the relators of the presentation are represented by using $g$ and $h$ as follows. 
\begin{align}
&  \xi^{-1}\gamma^{-1}\beta^{-1}\alpha^{-1}\xi\alpha\beta\gamma \notag \\
& \qquad = h^{-v-1}h^{-1}h^{v+1}h = 1 \notag \\
& \xi^{-1}\alpha\xi\gamma^{-1} \notag \\
& \qquad = h^{-v-1}h^{v+1}g^{-1}g = 1 \notag \\
& \gamma\xi^{-1}\beta\xi\gamma^{-1}\psi\alpha^{-1}\psi^{-1} \notag \\
& \qquad = g^{-1}gh^{-2v-1}gg(h^{-v}g)^ugh^{-v-1}(h^{-v}g)^{-u} \label{eq1} \\
&  \gamma\xi^{-1}\gamma\xi\gamma^{-1}\psi\beta^{-1}\psi^{-1} \notag \\
& \qquad = g^{-1}g^{-1}h^{v+1}g(h^{-v}g)^ug^{-1}h^{2v-1}g^{-1}(h^{-v}g)^{-u} \label{eq2} \\
&  \psi^{-1}\beta^{-1}\alpha^{-1}\psi\alpha\beta  \notag \\
& \qquad = (g^{-1}h^v)^ug^{-1}h^v(h^{-v}g)^uh^{-v}g = 1 \notag
\end{align}
We note that Equation~\eqref{eq1} is equivalent to 
\[
h^{-v-1}g^2(h^{-v}g)^{u-1}h^{-v}g^2h^{-v-1}(g^{-1}h^v)^{u-1}g^{-1} 
\]
and Equation~\eqref{eq2} is equivalent to 
\[
g(h^{-v}g)^{u-1}h^{v+1}g^{-2}h^v(g^{-1}h^v)^{u-1}g^{-2}h^{v+1}\;.
 \]
It then follows that Equations~\eqref{eq1} and \eqref{eq2} are equivalent by considering the inverse. 
Also note that the last equation corresponds to the commutativity of the meridian and longitude for the component $l_2$, and so, can be obtained from the other relations. 

Therefore, we have the following presentation of the knot group of $K$, 
\[
\langle g, h \ \vert \ 
g^2(h^{-v}g)^ugh^{-v-1}(h^{-v}g)^{-u}h^{-2v-1} \rangle 
\]
and is equivalent to the following. 
\[
\langle g, h \ \vert \ 
h^{v+1} g^{-1} ( h^{-v} g )^{-u} g^{-2} h^{2v+1} ( h^{-v} g )^{u}
\rangle 
\]
Furthermore, by setting $a=g^{-1} h^{v+1}$ and $b= ha^{-1}$, 
i.e., $ h = b a$ and $ g = h^{v+1} a^{-1} = ( b a )^{v+1} a^{-1} = (ba)^v b $, 
we have the following. 
\[
\langle a, b \ \vert \ 
( b a )^{v+1} a ( b a )^{-v-1} b^{-u-1} ( b a )^{-v} a ( b a )^v b^u 
\rangle
\]
Consequently the knot group $\pi_1(S^3-K)$ admits the presentation 
\[
\langle a, b \ \vert \ (w_1 a^m w_1^{-1}) b^{-r} (w_2^{-1} a^n w_2) b^{r-k} \rangle
\]
with $m=n=1$, $r=u+1$, $k=1$, $w_1 = (ba)^{v+1}$, $w_2=(ba)^v$. 

By Lemma \ref{lem}, a meridian of $K$ is represented by $\alpha$. 
This $\alpha$ is represented by 
$h^{v+1}g^{-1}$ which equals to $a$.
Also by Lemma \ref{lem}, a longitude of $K$ is represented by $\xi\xi\gamma^{-1}\psi\xi\gamma^{-1}\psi$.
This element is represented by 
$h^{v+1}(g(h^{-v}g)^u)^2$ which equals to $(ba)^{v+1}(ba)^vb^{u+1}(ba)^vb^{u+1}$. 
Adding $a^{-1}$, we have 
$a^{-1}(ba)^vb^{u+1}(ba)^vb^{u+1}(ba)^{v+1}$. 
Finally, as noted in the paragraph just below \cite[Proposition 3.2]{CW}, 
the presentation of the preferred longitude of $K$ is obtained 
from the above by adding $a^{-3(3v+2)-2u}$ as follows. 
\[
a^{-3(3v+2)-2u-1} ( ( ( b a )^v b^{u+1} )^2 ( b a )^v b ) a 
\]
Therefore the preferred longitude of $K$ is represented as $a^{-s} w a^{-t}$ with 
$s=2u+3(3v+2)+1$, $t=-1$ and $w = ((ba)^v b^{u+1} )^2 (ba)^v b$. 

This completes the proof of Proposition \ref{prop} since $u$ is any integer.
\end{proof}

We remark that the condition for the presentation of the preferred longitude such as $w$ is a positive word in \cite[Theorem 1.1]{IT} holds in the case of $u=-1$, although it does not hold in the case of $u \le -2$.

\end{document}